\documentclass[11pt]{article}
\usepackage{color,xcolor,comment}
\usepackage{graphicx,mathrsfs}
\usepackage{amssymb,bbm}
\usepackage{amsthm}
\usepackage{amsmath}
\usepackage[margin=1in]{geometry}
\usepackage{caption,numbysec}
\usepackage{subcaption}
\usepackage{float}
\usepackage{mdframed,pdfsync,mathrsfs}
\usepackage{esint}
\usepackage{dsfont}
\usepackage[normalem]{ulem}

\usepackage[colorlinks=true,pdfpagemode=UseNone,urlcolor=blue,linkcolor=blue,citecolor=blue,breaklinks=true]{hyperref}

\usepackage{cleveref}

\newcommand{\norm}[1]{\lVert#1\rVert}
\newcommand{\one}{{\mathbbm{1}}}
\usepackage{enumitem}

\newtheorem{thm}{Theorem}[section]
\newtheorem{lemma}[theorem]{Lemma}

\newtheorem{remark}[theorem]{Remark}
\newtheorem{cor}[theorem]{Corollary}

\newcommand{\bea}{\begin{eqnarray*}}
\newcommand{\eea}{\end{eqnarray*}}
\newcommand{\ben}{\begin{eqnarray}}
\newcommand{\een}{\end{eqnarray}}
\newcommand{\beq}{\begin{equation}}
\newcommand{\eeq}{\end{equation}}

\newcommand{\R}{\ensuremath{\mathbb{R}}}

\newcommand{\Rm}{{\mathbb R}}

\newcommand{\Sm}{\ensuremath{\mathbb{S}}}

\newcommand{\bal}{\begin{aligned}}
\newcommand{\enbal}{\end{aligned}}
\newcommand{\be}{\begin{equation}}
\newcommand{\ee}{\end{equation}}

\renewcommand{\hat}[1]{\widehat{#1}}

\newcommand{\farc}{\frac}
\DeclareMathOperator\erfc{erfc}

\renewcommand{\d}{\partial}


\title{Reaction enhancement by flux-limited chemotaxis}
\author{Jing An
\thanks{School of Mathematical Sciences, Shanghai Jiao Tong University, Shanghai 200240, China; email: jingan1009@sjtu.edu.cn}
\and
Alexander Kiselev
\thanks{Department of
Mathematics, Duke University, Durham, NC 90320, USA; email: kiselev@math.duke.edu}
\and
Yao Yao
\thanks{Department of Mathematics, National University of Singapore, Singapore 119076, Singapore; email: yaoyao@nus.edu.sg}}
		
\begin{document}

\maketitle
\numberbysection

\begin{abstract}
Chemotaxis plays a crucial role in a variety of processes in biology and ecology. Quite often it acts to improve efficiency of biological reactions;
one example is the immune system signalling, where infected tissues release chemokines attracting monocytes to fight invading bacteria.
Another example is reproduction, where eggs release pheromones that attract sperm.
In this paper, we analyze a system of two reacting densities, one of which is chemotactic on another. Since the speed of any biological agents is limited,
we employ flux limited chemotaxis model. Our main result is the rigorous derivation of the scaling laws showing how presence of chemotaxis affects
the typical reaction time scale. This work builds on the results of \cite{kiselev2022chemotaxis}, which employed a classical Keller-Segel chemotaxis
term (not flux limited) -  leading to the effect of possible over concentration and restricting the results to radial data.
The model presented here is more reasonable biologically and covers broader parameter regimes.
\end{abstract}

\section{Introduction}

\subsection{The background}

Chemotaxis describes the motion of cells or species that sense and attempt to move
towards higher (or lower) concentration of some chemical. The first mathematical studies of chemotaxis go back to works of Patlak  \cite{Patlak} and
Keller-Segel \cite{KS1}, \cite{KS2}. In these papers, the focus is on the system of equations describing mold or bacteria that release signalling attractive chemical
helping them to aggregate. In the simplest form, this system can be reduced to a single parabolic-elliptic equation
\begin{equation}\label{chemo1}
\partial_t \rho - \Delta \rho +
\chi \nabla\cdot(\rho \nabla (-\Delta)^{-1} \rho)=0, \,\,\,\rho(x,0)=\rho_0(x).
\end{equation}
Here the density of the chemical $c$ is taken to be just $(-\Delta)^{-1}\rho,$ a reasonable model if one assumes that the chemical is produced and diffuses faster than other
relevant time scales in the problem. There has been much research on the equation \eqref{chemo1} and related models, in particular driven by dramatic analytic properties of solutions:
in dimensions $d \geq 2$ singularities can form in a finite time (see e.g. \cite{Pert} where more references can be found).

However, in many settings in biology where chemotaxis is present, it facilitates and enhances
success rates of reaction-like processes. In these settings, typically one species releases a signalling chemical,
while another is chemotactic on it; more than two different densities can be involved but we are going to focus on
the fundamental model.
One example is reproduction, where eggs secrete chemicals that attracts sperm
and help improve fertilization rates. This is especially well studied for marine life such as corals, sea urchins, mollusks,
etc (see \cite{HRZZ,RZ,ZR} for further
references), but the role of chemotaxis in fertilization extends to a great number of
species, including humans~\cite{Raltetal}.
In the same vein, many plants appeal primarily to the insects' sense of smell to attract pollinators and facilitate
fertilization. Another process where chemotaxis plays an important role is mammal immune systems fighting infections.
Inflamed tissues release special proteins, called chemokines,
that serve to chemically attract monocytes, blood killer cells, to the source
of infection \cite{Desh}, \cite{Taub}. In all these settings chemotaxis serves to enhance processes that can be
thought of as absorbing reactions between different species - the product is qualitatively different, as in the case of fertilization
or immune cells fighting invaders.

To the best of our knowledge, there has been limited mathematical research to model this kind of processes. The interesting question in this
context is how much the presence of chemotaxis affects the reaction rates? What is the influence of geometry of the initial data
as well as the structure of the ambient space (in the simplest case, dimension)? The evidence that the role of chemotaxis may be
quite essential was indirectly provided in \cite{ccw,chw,ds}, where models based on reaction-diffusion and on reaction-diffusion with advection
were shown to strongly under predict the fertilization success rates of marine animals observed in nature and lab experiments (see e.g. \cite{Lasker,Penn, Yund}). As far as we know, the first step in the direction of rigorous evaluation of the effect of chemotaxis on reaction rates was made in \cite{KR1,KR2}, where
a single density equation \eqref{chemo1} but with absorbing reaction and fluid advection was considered. The result showed significant impact of
chemotaxis on reaction rates, especially in the setting of weak reaction that is most natural from the biological point of view. Yet the single
density model does not reflect the fact that typically there are more than one species in the system, where only one species is chemotactic on the other.
In \cite{kiselev2022chemotaxis}, the authors considered a model
\begin{equation}
\begin{aligned}
&\partial_t \rho_1 - \kappa \Delta \rho_1 + \chi \nabla\cdot(\rho_1 \nabla(-\Delta)^{-1} \rho_2) = -\epsilon \rho_1 \rho_2\\
&\partial_t \rho_2 = -\epsilon \rho_1 \rho_2.
\end{aligned} 
\label{eq:before_normalization}
\end{equation}
with the aim of establishing estimates on influence of chemotaxis on reaction rates in the system setting.
The initial data was given by a stationary density $\rho_2$ supported in a unit (after rescaling) ball with the center at the origin while the
density $\rho_1$ with much larger mass (as is typical in applications) was supported at a distance $L$ away.
However, a modelling issue came up: the classical Keller-Segel model where the chemotaxis-induced drift is proportional to the gradient of the signalling chemical $c=(-\Delta)^{-1}\rho_2$ can potentially lead to over concentration of the density $\rho_1$ at the center of the support of $\rho_2.$
This is due to large size of $\chi \nabla c$ near the center of the support of $\rho_2.$ As a result, the supports of $\rho_1$ and $\rho_2$ may have reduced overlap
even once the density $\rho_1$ has been transported towards the origin, and this may result in artificially slow reaction rate. For this reason,
the consideration in \cite{kiselev2022chemotaxis} has been limited to the radial case, where one can be assured that $\rho_1$ will react with a large
portion of $\rho_2$ while passing through towards the center of $\rho_2$ support.

Of course, in any realistic system the speed of transport of biological agents is limited. In recent years, there has been much work
on chemotaxis models with limited size of drift - the so-called flux-limited chemotaxis (see e.g. \cite{BW,BBTW,CKWW,CPY,DS,JV1,HP,HPS,PVW,Winkler2021,Winkler2021a}) for further references.
In particular, papers \cite{DS,JV,PVW} provided derivation of the flux limited Keller-Segel system from kinetic models built on biologically
reasonable assumptions about behavior of the modeled organisms.

Our goal in this paper is to analyze a two-dimensional flux-limited chemotaxis system, in which $\rho_1\geq 0$ represents the chemically attracted density, and $\rho_2\geq 0$ denotes the density of attractant staying put at the center:
\begin{equation}\label{main_old}
\begin{aligned}
&\d_t \rho_1 - \tilde{\kappa} \Delta \rho_1 +\nabla\cdot\Big(\rho_1 \frac{\nabla c}{|\nabla c|}\Psi(|\nabla c|)\Big) = -\tilde{\epsilon} \rho_1 \rho_2\\
&\d_t \rho_2 = -\tilde{\epsilon} \rho_1 \rho_2, \quad\quad -\tilde{\sigma} \Delta c = \tilde{a}\rho_2.
\end{aligned}
\end{equation}
(here, we are using the parameters with tilde since we reserve the simpler notation for the parameters we will obtain after rescaling).
The drift term involving a cut-off function models the speed limit on how fast biological agents can move in reality.
We denote this sensitivity cut-off function $\Psi$. We are going to assume a simple form of $\Psi$ that is close to piecewise
linear. This allows us to essentially describe this function with two parameters $\tilde{v}_0$ and $\tilde \chi$.
To be precise, we define $\Psi: \R_+\to \R_+$ to be
a $C^\infty$ increasing function, $\Psi(0) = 0$, and furthermore (see Fig. \ref{fig:psi})
\begin{align}\label{cut-off}
    \Psi(z) = \begin{cases}
    \tilde{\chi} z, &\text{ for $0\leq z\leq \tilde \beta-\tilde \delta$}\\
    \psi(z), &\text{for $\tilde\beta-\tilde\delta<z< \tilde\beta$}\\
    \tilde{v}_0 &\text{ for $z\geq  \tilde\beta$}.
    \end{cases}
\end{align}
Near the origin, $\Psi$ grows linearly with the slope $\tilde \chi$, and then it gets capped by speed limit $\tilde{v}_0$. In order to make $\Psi\in C^\infty$, we introduce a sufficiently small $\tilde\delta>0$ such that $\tilde\delta \ll \tilde\beta$. In the buffer interval $(\tilde\beta-\tilde\delta, \tilde\beta)$, we use a smooth increasing function $\psi(z)\in C^\infty$ to connect the linear parts; we can arrange it so that $\psi(z) \geq \tilde \chi z$ and $\tilde\beta \tilde \chi = \tilde v_0.$    
\begin{figure}[h!]
\centering
\includegraphics[scale=1]{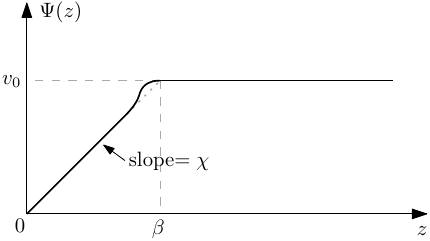}
 \caption{An illustration of the cut-off function $\Psi$ after performing the rescaling in Section~\ref{sec_rescale}. Here $\beta>0$ is the sensitivity, $v_0$ is the saturation value, and $\chi$ is the slope.}
 \label{fig:psi}
\end{figure}


We assume that the support of $\rho_1$ is initially located inside $|x| \sim \tilde{L}$ region. More precisely, we will suppose that the mass distribution of $\rho_1$ satisfies
\begin{equation*}
    \int_{|x|\leq \tilde{L}}\rho_1(x,0) dx = \tilde{M}_0,\quad \int_{|x|\leq \tilde{L}/2}\rho_1(x,0) dx \ll \tilde{M}_0.
\end{equation*}
The initial condition for $\rho_2$ has the form $\rho_2(x,0) = 2\theta \eta(x)$, with some $\theta>0$, and some $C_0^\infty$ function $\eta(x)$ that we assume for simplicity to be close to $\one_{B(0,\tilde{l})}(x)$, with $\one_{B(0,(1-\delta_0)\tilde{l})}(x)\leq \eta(x)\leq \one_{B(0,\tilde{l})}(x)$, where $\delta_0$ is sufficiently small. Here $\one_{B(0,\tilde{l})}(x)$ is the characteristic function of the ball with a radius $\tilde{l}>0$ centered at the origin. 

Driven by the chemo-attractant $c$ produced by $\rho_2$ near the origin, the substance $\rho_1$ will drift-diffuse towards stationary $\rho_2$ and react. The purpose of this paper is to study the reaction rate, namely the rate of decay of
    $\int_{\R^2} \rho_2(x,t) dx,$
which can be measured by ``half-time" $\tau:$ the time period over which approximately half of the initial mass of $\rho_2$ is consumed by the reaction.


\subsection{Rescaling}
\label{sec_rescale}

We will now rescale the chemotaxis system (\ref{main_old}) in time and space, so that we can simplify the model, reduce the number of parameters, and make the remaining ones dimensionless.
Let us first consider physical dimensions of parameters. We will write the dimension of length as $[L]$ and the dimension of time as $[T]$. Then we have
that $\tilde{\kappa}$ is of the dimension $[L^2 T^{-1}]$, $\tilde{\chi}$ is of dimension $[L^2 T^{-1}]$, $\tilde{v}_0$ is of dimension $[L T^{-1}]$, $\tilde \epsilon$ is of dimension $[T^{-1}]$, $\tilde a$ is of dimension $[T^{-1}]$, and $\tilde \sigma$ is of dimension $[L^2 T^{-1}]$.
The sensitivity $\tilde{\beta}$ that measures how quickly the saturation of the level of $\nabla c$ is reached, satisfies the relation $\tilde \chi \tilde \beta = \tilde{v}_0$, so its dimension is $[L^{-1}]$. Note that the chemical $c$ and densities  $\rho_1, \rho_2$ are dimensionless.
Take  $x' = x/\tilde{l}, ~t' = t\tilde{\kappa}/\tilde{l}^2$. We define the new parameters by $\chi =\tilde{\chi}/\tilde{\kappa},~v_0 = \tilde{v}_0 \tilde{l}/\tilde{\kappa},~\epsilon=  \tilde{\epsilon} \tilde{l}^2/\tilde{\kappa},~M_0 = \tilde{M}_0/\tilde{l}^2, L=\tilde{L}/\tilde{l}, \beta = \tilde{\beta} \tilde{l}$, $\sigma = \tilde{\sigma}/(\tilde{l}^2 \tilde a),$ and we also introduce $\gamma = \theta \chi/\sigma = \theta \tilde{\chi}/\tilde{\kappa}$ that will be a useful parameter to state the results.
By rescaling and the initial condition setup, we have the initial assumptions as
\begin{equation}\label{ic}
\begin{aligned}
& \int_{|x| \leq L} \rho_1(x,0)\,dx = M_0,  \,\,\,\,\, \int_{|x| \leq L/2} \rho_1(x,0)\,dx \ll M_0; \\
& 2\theta \one_{B(0, 1-\delta_0)}(x) \leq \rho_2(x,0) \leq  2\theta \one_{B(0, 1)}(x), \\
 & \text{and}\quad \frac{7\pi\theta}{4} \leq \|\rho_2(\cdot, 0)\|_1\leq 2\pi \theta,
\end{aligned}
\end{equation}
where in the last line we assumed that $\delta_0$ is sufficiently small to ensure the lower bound on $\|\rho_2(\cdot,0)\|_1.$
After rescaling, the cut-off function's slope and maximum are both changed as above and become dimensionless.
We still denote the resulting function $\Psi,$ and we will go back to using $x,t$ instead of $x',t'$ for the rescaled variables.
We have the following simplified model to investigate:
\begin{equation}\label{main}
\begin{aligned}
&\d_t \rho_1 - \Delta \rho_1 +\nabla\cdot\Big(\rho_1 \frac{\nabla c}{|\nabla c|}\Psi(|\nabla c|)\Big) = -\epsilon \rho_1 \rho_2\\
&\d_t \rho_2 = -\epsilon \rho_1 \rho_2, \quad\quad -\sigma\Delta c = \rho_2.
\end{aligned}
\end{equation}
If we remove the flux limitation by taking $\Psi(z) = \chi z$, then we recover the chemotaxis system that was analyzed in \cite{kiselev2022chemotaxis}. The results in \cite{kiselev2022chemotaxis} show quantitatively
how the chemotaxis term enhances the reaction rate, and significantly reduces the half-time in the parameter regime where $M_0 \epsilon\gg \gamma \gg 1$ and $M_0\gg \theta$. This regime is definitely relevant, for instance, in marine animal broadcast spawning \cite{kiselev2022chemotaxis}.
However, the results obtained in \cite{kiselev2022chemotaxis} have shortcomings. One major constraint is that, as we mentioned above, the initial conditions are assumed to be radially symmetric.
This is definitely not reasonable in applications. For the flux-limited chemotaxis, we will remove this radial limitation. We are also able to extend the results to new parameter regimes, in particular to the ``risky reaction" case $M_0 \epsilon \ll 1.$

\subsection{Main result}

In this paper, we will make certain assumptions on the parameters of the problem. The first one is that $v_0 \leq 1.$ In the case where $v_0 \geq 1$ one can also obtain estimates on the reaction half-time, but the effect of over-concentration
similar to the classical Keller-Segel chemotaxis term may be present for large $v_0.$ We focus on the regime $v_0 \leq 1$ as it is more interesting to us, being relevant in at least some applications - such as broadcast spawning
of marine animals. Indeed, see \cite{chemoshear} for discussion of values of different parameters that, given our definition $v_0 = \tilde{v}_0 \tilde{l}/\tilde{\kappa},$ indicates that $v_0 \leq 1.$

The second assumption that we are going to make is that $\gamma$ is sufficiently large, $\gamma \gg 1$ (specifically, $\gamma \geq 16$ would work for the conclusions of the main theorem below). For small $\gamma,$ it is not clear that chemotaxis will offer a significant
improvement in reaction rates compared to diffusion. In fact, the local mass equation \eqref{eq:M} that we will use for transport estimates suggests that chemotaxis effects may be more subtle and perhaps transient if $\gamma$ is not sufficiently large.

Finally, we will also assume that $M_0 \gg \theta,$ and in fact that $M_0 v_0^2 \gg \theta.$ This is a natural assumption in many applications, in particular in reproduction where the size of $M_0$ is typically dominant.

The techniques that we develop should be useful in analyzing other parameter regimes as well, but we focus on the above three natural assumptions in order to avoid considering too many cases and making the
paper unwieldy.

Let us define the ``half-time" $\tau$ that we will use throughout the paper by
\begin{equation}\label{halftime}
    \tau:=\sup\Big\{t\geq 0: \int_{\R^2} \rho_2(x,t) dx \geq \pi \theta \Big\}.
\end{equation}
Then the main result of this paper is as follows.
\begin{thm}\label{thm:main}
    Assume that the initial conditions $\rho_1(x,0)$ and $\rho_2(x,0)$ are as described in \eqref{ic}.
    Suppose that $v_0 \leq 1,$ $\gamma \gg 1,$ and $M_0 v_0^2 \gg \theta.$
Let  $R_0 = \frac{\gamma}{2v_0}-1.$
Then the half-time $\tau$ for the flux-limited system \eqref{main} satisfies
\be\label{mainres}
\tau \leq \left\{ \begin{array}{ll}
C \left( \frac{1}{v_0^2} + \frac{1}{\epsilon v_0^2 M_0} \right), & L \leq v_0^{-1} \\
C \left( \frac{L}{v_0} + \frac{1}{\epsilon v_0^2 M_0} \right), & v_0^{-1} \leq L \leq R_0 \\
C \left( \frac{L^2}{\gamma}+ \frac{\gamma}{v_0^2} + \frac{1}{\epsilon v_0^2 M_0} \right), & R_0 \leq L. \\
\end{array} \right.
\ee
\end{thm}
\begin{remark}
Note that under our assumptions, $R_0 \geq v_0^{-1} \geq 1.$ In the $L \geq R_0$ case, we could replace $L^2/\gamma$ in the estimate for $\tau$ with $(L-R_0)^2/\gamma,$ but since $R_0^2/\gamma \sim \gamma/v_0,$ this does not lead to a nontrivial improvement
and we choose a simpler formula.
\end{remark}

\begin{remark}
The estimate \eqref{mainres} agrees with that suggested by heuristic arguments.
In \eqref{mainres}, $L^2/\gamma$ is the time of transport from $|x| \sim L$  to $|x| \sim R_0$ that is driven by chemotaxis before the speed limit $v_0$ is reached. The times $\frac{\gamma}{v_0^2} = \frac{R_0}{v_0}$ and $\frac{L}{v_0}$
correspond to transport from $|x| \sim R_0$ or $|x| \sim L$ to $|x| \sim v_0^{-1}$, and in this range chemotaxis produces drift $\sim v_0.$ At the distance $|x| \sim v_0^{-1}$ diffusion and chemotaxis balance - it will be explained more clearly in the beginning of Section~\ref{R0transport}.
The time $v_0^{-2}$ corresponds to diffusive equilibration in the region $|x| \sim v_0^{-1}.$ Finally, $\frac{1}{\epsilon v_0^2 M_0}$ is the reaction time. The factor $v_0^{-2}$ appears because the mass of the density $\rho_1$ of the order $M_0$ will be spread over
the ball of radius $\sim v_0^{-1}.$
\end{remark}

\begin{remark}
In what follows, we will focus on the $R_0 \leq L$ case, since the argument for this case includes all the techniques, estimates and steps necessary to handle the other options, with minor adjustments.
In order to not make the paper too technical, we leave these adjustments to interested readers.
\end{remark}

To contrast \eqref{mainres} with the pure reaction-diffusion, it was shown in \cite{kiselev2022chemotaxis} that when chemotaxis is absent, the purely reaction-diffusion half-time
$\tau_D$ satisfies $\tau_D \gtrsim L^2 /\log (M_0 \epsilon)$ when $M_0 \epsilon \gg 1$
and $\tau_D \gtrsim L^2 e^{\frac{C}{M_0 \epsilon}}$ when $M_0 \epsilon \ll 1.$
Thus it is clear that in all cases, chemotaxis can reduce the half-time significantly when $\gamma$ is large. In addition, there is a dramatically improved
dependence of the reaction rate on $M_0 \epsilon$ in the case of a ``risky" reaction where $M_0 \epsilon \ll 1.$

\section{Transport: problem reformulation}

The main part of the argument will focus on understanding the transport of the density $\rho_1$ towards the
support of $\rho_2.$ For this purpose,
let us consider the local mass $M(r,t) := \int_{B(0,r)} \rho_1(x,t) dx.$ Using (\ref{main}), one can compute that it satisfies the following equation:
\begin{equation}\label{eq:M}
\partial_t M = \partial_{rr} M - \frac{1}{r} \partial_r M -  \int_{\partial B(0,r)} \rho_1 \Psi(|\nabla c|)\frac{\nabla c}{|\nabla c|}\cdot \hat n  dS(x) -   \int_{B(0,r)} \epsilon \rho_1 \rho_2 dx,
\end{equation}
with $\hat{n}=x/|x|$ the unit normal vector. We will consider evolution of the local mass  $M(r,t)$ for $t\in[0,\tau]$, where $\tau$ is the reaction half-time (\ref{halftime}). Instead of directly dealing with the nonlinear flux-limited chemotaxis term in (\ref{eq:M}),
in Section~\ref{Hsec} we will construct a radially symmetric function $H(x)$ that satisfies
\begin{equation}\label{def:H}
\partial_r H(r) \geq  \sup_{\partial B(0,r)} \Psi(|\nabla c|)\frac{\nabla c\cdot \hat n}{|\nabla c|}, \quad \text{for}~0\leq t\leq \tau~~\text{and for any }r>0.
\end{equation}
This will effectively linearize the chemotactic drift.

Right now, we want to simplify the problem by absorbing the nonlinear reaction term. Let $h(t) := \int_{\mathbb{R}^2} \epsilon \rho_1(x, t) \rho_2(x,t) dx$. Integrating $\rho_2$ equation in (\ref{main}) in space, we obtain
\begin{equation}
    \d_t \int_{\R^2} \rho_2(x,t) dx = -h(t).
\end{equation}
Integrating in time from $0$ to $\tau$, we get  that
\begin{equation}
    \int_0^{\tau} h(s) ds = \int_{\R^2} \rho_2(x,0) dx-\int_{\R^2} \rho_2(x,\tau) dx\leq \pi \theta,
\end{equation}
using the definition of $\tau$.

With $H(x)$ and $h(t)$ introduced, and since $\partial_r M \geq 0,$ equation \eqref{eq:M} implies
\begin{equation}
\partial_t M \geq \partial_{rr} M - \frac{1}{r} \partial_r M - \partial_r M \,\partial_r H -  h(t).
\end{equation}
Let $F(r,t) := M(r,t) + \int_0^t h(s)ds$, then $F$ satisfies
\begin{equation}
\begin{split}
\partial_t F &= \partial_t M + h(t)\geq \partial_{rr} F - \frac{1}{r} \partial_r F - \partial_r F \, \partial_r H \quad\text{ for all }t\in[0,\tau].
\end{split}
\label{ineq:tilde_M}
\end{equation}

We introduce an auxiliary Fokker-Planck equation to control $F(r,t)$,
\begin{equation}
\partial_t u - \Delta u + \nabla \cdot (u \nabla H) = 0,
\label{eq:fokker_planck}
\end{equation}
with the initial data $u_0 = \rho_1(x,0).$
Consider $M_u(r,t) := \int_{B(0,r)} u(x,t) dx$, then from (\ref{eq:fokker_planck}) $M_u$ satisfies the following equation
\begin{equation}\label{ineq:M_u}
\partial_t M_u = \partial_{rr} M_u - \frac{1}{r} \partial_r M_u - \partial_r M_u \,\partial_r H
\end{equation}
with initial condition $M_u(r,0) = F(r,0)$. We apply the comparison principle  to equations \eqref{ineq:tilde_M} and \eqref{ineq:M_u}, and it immediately yields that $F(r,t) \geq M_u(r,t)$. In other words,
\begin{equation}
M(r,t) \geq M_u(r,t) - \int_0^t h(s) ds \geq  M_u(r,t) - \pi \theta, \quad\text{ for all }t\in[0,\tau].
\label{ineq:rho_vs_u}
\end{equation}
This series of reformulations helps reduce the problem to analyzing the Fokker-Planck equation (\ref{eq:fokker_planck}). If we can obtain a good lower bound for $M_u(r,t)$ from (\ref{ineq:M_u}), then
using the $M_0 \gg \theta$ assumption we can estimate the time when $M(r,t)$ becomes sufficiently large for small $r$. This implies
that the substance $\rho_1$ has been transported to the center where $\rho_2$ is located and it will remain to estimate the local reaction time between the two densities.

\section{Construction of $H$}\label{Hsec}
We now construct $H$ such that (\ref{def:H}) holds, given the form of the sensitivity cut-off function \eqref{cut-off}.

\begin{lemma}\label{lem418}
Suppose that $\gamma$ is sufficiently large so that $R_0 = \frac{\gamma}{2v_0} -1 >1.$ In order to satisfy \eqref{def:H}, we can define $H'(r)$ as a smooth function for $r>0$ that is given by
\begin{align}\label{defH48}
\partial_r H (r)= \left\{ \begin{array}{ll}
 -\frac{\gamma}{4r}, & r > R_0 \\
 \eta_1(r), & R_0 \geq r \geq R_0-1 \\
 -v_0, & R_0-1 \geq r > 1 \\
 \eta_2(r), & 1 \geq r > r_0 \\
 v_0, & r_0 \geq r \geq 0.
\end{array}
\right.
\end{align}
See Figure~\ref{fig_H_der} for an illustration of $\partial_r H(r)$.
Here $r_0 \in (1/\sqrt{2},1)$ depends on $\gamma$ and will be specified later, while $\eta_1$ and $\eta_2$ are smooth functions connecting different transport regimes to make $H(r)$ smooth.
The function $\eta_1$ can be chosen to satisfy $-\frac{\gamma}{4r} \geq \eta_1(r) \geq -v_0,$ and the function $\eta_2$ can be chosen to satisfy $v_0 \geq \eta_2(r) \geq -v_0.$
\end{lemma}

\begin{figure}[h!]
\centering
\includegraphics[scale=1]{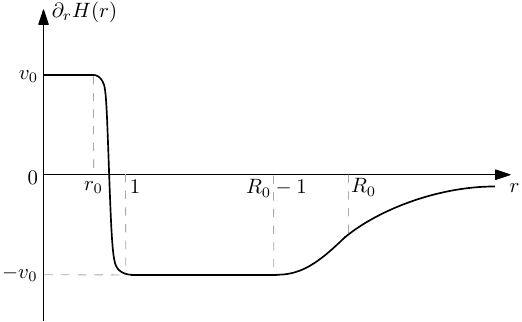}
 \caption{An illustration of $\partial_r H(r)$.}
 \label{fig_H_der}
\end{figure}

With the above definition of $\partial_r H(r)$, integrating in $r$, we can define $H(r)$ as
\be\label{def:H_2}
H(r) = \left\{ \begin{array}{ll}
 -\frac{\gamma}{4}\log r, & r > R_0 \\
 -\int_r^{R_0} \eta_1(s)\,ds - \frac{\gamma }{4}\log R_0, & R_0 \geq r > R_0 -1 \\
 -v_0r + v_0(R_0-1) -\int_{R_0-1}^{R_0} \eta_1(s)\,ds - \frac{\gamma }{4}\log R_0, & R_0-1 \geq r > 1 \\
 -\int_r^1 \eta_2(s)\,ds + v_0(R_0-2)  -\int_{R_0-1}^{R_0} \eta_1(s)\,ds - \frac{\gamma }{4}\log R_0, & 1 \geq r \geq r_0 \\
  -\int_{r_0}^1 \eta_2(s)\,ds + v_0(r-r_0+R_0-2)  -\int_{R_0-1}^{R_0} \eta_1(s)\,ds - \frac{\gamma }{4}\log R_0, & r_0 \geq r \geq 0.
\end{array}
\right.
\ee

We note that, as will follow from the proof, the form of $H$ can be made simpler if we take it to be just Lipschitz regular. But dealing with the Fokker-Planck equation with rough drift
makes the argument more technical, and we prefer to introduce more complex definition that has the advantage of leading to $H$ that is smooth away from zero and with simple Lipschitz behavior at the origin.

\begin{proof}
Let us consider the following family of functions
\begin{equation}\label{g_class}
    \mathcal{M} = \Big\{ g\geq 0: \text{supp }g \subset B(0,1),\,\,\, \|g\|_\infty \leq 2\theta,\,\,\,2\pi \theta \geq \|g\|_1 \geq \pi\theta\Big\}.
\end{equation}
Because we assume that $\|\rho_2(\cdot, t)\|_1 \geq \frac{1}{2} \|\rho_2(\cdot, 0)\|_1 \geq \pi\theta$ for $t\in[0,\tau]$, and 
since $\rho_2(x,t)\leq\rho_2(x,0) $ for all $x$, we have that $\rho_2(\cdot, t) \in \mathcal{M}$ for all $t\in[0,\tau]$.
Furthermore, notice that the family $\mathcal{M}$ has rotational invariance (i.e. $g\in \mathcal{M} \Leftrightarrow Tg \in \mathcal{M}$, where $T$ is a rotation). Hence to find an $H$ that satisfies \eqref{def:H}, it suffices to find an $ H$ such that
\begin{align}\label{radial_comp_2d}
    \d_r H(r) \geq  \sup_{g\in \mathcal{M}}  \Psi \left(\left|\frac{1}{\sigma}\frac{\d}{\d x}  ((-\Delta)^{-1} g)(r,0)\right|\right)\frac{\frac{\d}{\d x} ((-\Delta)^{-1} g)(r,0)}{|\frac{\d}{\d x}  ((-\Delta)^{-1} g)(r,0)|},\quad\text{ for all }r>0.
\end{align}

For any point $(x,y)\in \mathbb{R}^2$, we define $V(x, y; r)$ to measure how a unit mass concentrated at $(x,y)$ affects the drift in the horizontal direction at the point $(r,0)$, namely
\begin{equation}
    V(x, y; r) :=  \frac{\partial}{\partial x} \big((-\Delta)^{-1} \delta_{(x, y)}\big) (r,0).
\end{equation}
 Then estimating the right side of (\ref{radial_comp_2d}) reduces to estimating the quantity
 \begin{equation}\label{jun201}
   \int_{\mathbb{R}^2} g(x, y) V(x, y; r)~ dxdy, \quad\text{ for all }r>0.
 \end{equation}
In particular, the magnitude of the expression in (\ref{jun201}) determines if the saturation level of the sensitivity function $\Psi$ has been reached.

Let us digress a little and take a closer look at the function $V(x,y; r)$. First, a direct computation gives that
\begin{equation}
    V(x, y; r) = \frac{1}{2\pi}\frac{x-r}{(x-r)^2 + y^2}.
\end{equation}
We fix $r>0$, and consider $V(x,y; r)$ as a function of $x$ and $y$. Let us investigate the level sets of $V(x,y; r)$. Elementary algebraic manipulations yield that the level sets of $V(x,y;r)$ are all circles passing through the point $(r,0)$; more precisely, for any constant $b \ne 0$, the level set $V(x,y;r)=b$ is given by a circle centered at $(r+\frac{1}{4\pi b}, 0)$ with radius $\frac{1}{4\pi |b|}$ since
\begin{equation}
  \frac{x-r}{(x-r)^2 + y^2} = 2\pi b \quad \Leftrightarrow \quad \left(x-r-\frac{1}{4\pi b}\right)^2+y^2 = \left(\frac{1}{4\pi b}\right)^2.
\end{equation}

\textbf{$\bullet$ Case 1:} $r>1$. Note that when $r>1$, the whole disk $B(0,1)$ lies in the set $\{(x,y):V(x,y;r)<0\}$. Based on the structure of the level sets $V(x,y;r)$, a larger radius of the level set circle implies a smaller $|b|$, i.e. a larger $b<0$.
Therefore, among all $(x,y)\in B(0,1)$, the maximum of $V(x,y;r)$ is achieved at the point $(-1,0)$, which corresponds to the value $b=-\frac{1}{2\pi(1+r)}$, and the minimum is achieved at $(1,0)$, which corresponds to the value $b=\frac{1}{2\pi(1-r)}$ .
Thus for any $g\in \mathcal{M}$, we have
\begin{equation}\label{aux126}
 \frac{\gamma}{\chi(1-r)}\leq \frac{\|g\|_1}{2\pi \sigma (1-r)}\leq \frac{1}{\sigma}\int_{\mathbb{R}^2} g(x, y) V(x, y; r)~ dxdy \leq -\frac{\|g\|_1}{2\pi \sigma (1+r)} \leq - \frac{\gamma}{2\chi(1+r)},
\end{equation}
because $\pi\gamma \sigma/\chi = \pi \theta \leq \norm{g}_1\leq 2\pi \theta = 2\pi\gamma \sigma/\chi$ by (\ref{g_class}).

Given our sensitivity function $\Psi$, we know that the right side of \eqref{radial_comp_2d} is bounded below by $-v_0$.
As $-\frac{\gamma}{2\chi(1+r)}$ is increasing for $r>1$, we can find a threshold value $R_0$ such that $\frac{\gamma}{2\chi(1+R_0)} = \frac{v_0}{\chi}$ (i.e. $R_0 = \frac{\gamma}{2v_0}-1$),
so that the magnitude of (\ref{jun201}) is guaranteed to hit the value of sensitivity $\beta = v_0/\chi.$
Then for $r \leq R_0$ we for sure reach the minimal drift $v_0$ towards the origin.
For $r \geq R_0,$ we observe that
\begin{equation}
    \frac{\gamma}{2\chi(1+r)} \geq \frac{\gamma}{4\chi r},
\end{equation}
and, multiplying by the slope $\chi$ of $\Psi,$ we may set
\[    \d_r H(r) = -\frac{\gamma}{4r} \quad\text{ for $r \geq R_0 $} \]
so that (\ref{def:H}) holds in this range.
If $R_0 \geq r >1,$ according to \eqref{aux126}, the definition of $R_0$, and the saturation property of $\Psi$ we can define $\d_r H = -v_0$ to satisfy \eqref{def:H}. To make $H$ smooth, we introduce the connecting function $\eta_1$ that,
according to its properties specified in Lemma~\ref{lem418}, only makes $\partial_r H$ larger.

\textbf{$\bullet$ Case 2 :} $r_0<r\leq 1$ (where $r_0$ such that $1/\sqrt{2}<r_0<1$ will be determined below).
In the regime $\sqrt{2}-1<r\leq 1$, one can construct an explicit $\tilde g\in \mathcal{M}$ to maximize $\int g(x,y) V(x,y;r) dxdy$. We will show that there exists an $r_0>\frac{1}{\sqrt{2}}$ such that
\begin{align}
   \frac{1}{\sigma}  \int g(x,y) V(x,y;r) dxdy\leq -\beta, \quad \text{ for all } g\in\mathcal{M}~\text{and}~ r_0<r\leq 1.
\end{align}
We note that $\tilde g$ must have a mass exactly $\pi\theta$ for the following reason. For any $g\in \mathcal{M}$, due to the restriction $\|g\|_\infty \leq 2\theta$, the mass of $g$ in the set $\{V>0\}\cap B(0,1)$ (which is the set $\{x>r\}\cap B(0,1)$) is strictly less than $\pi\theta$. Since $g$ has to satisfy $\|g\|_1 \geq \pi\theta$, this means the rest of the mass has to be allocated in the set $\{V<0\}$. As a result, the mass of the optimal $\tilde g$ must be as small as possible, which is $\pi\theta$.

To find $\tilde g$ that maximizes $\int gV$ and has mass $\pi\theta$, it suffices to set $\tilde g$ as $2\theta \one_\Omega$, where $\Omega =B(0,1) \cap \{V>b\} $, and $b<0$ is some constant to ensure that $\|\tilde g\|_1 = \pi\theta$. Making use of our previous discussion on the level sets of $V$, we have $\Omega = B(0,1) \cap B\big((r+\frac{1}{4\pi b}, 0), \frac{1}{4\pi |b|}\big)^c$.

In the case $\sqrt{2}-1<r\leq 1$, we can estimate the drift from above if we take $\Omega = B(0,1) \backslash B\big((r-\frac{1}{\sqrt{2}},0), \frac{1}{\sqrt{2}}\big)$. Indeed, we are removing exactly half of the disc's area where $V(x,y,r)$ is smallest. The constraint
$r > \sqrt{2}-1$ comes from the observation that for smaller $r,$ the area of the set $\Omega$ will be larger than half of the disc's area.

As a result, we compute
\begin{equation*}
\begin{split}
\frac{1}{\sigma}\int \tilde{g}(x,y) V(x,y;r) dxdy &= \frac{2\gamma}{\chi} \frac{\partial}{\partial x} \big((-\Delta)^{-1} \one_\Omega \big) (r,0)\\
&=\frac{2\gamma}{\chi} \left[  \frac{\partial}{\partial x} \big((-\Delta)^{-1}  \one_{B(0,1)} \big) (r,0) - \frac{\partial}{\partial x} \big((-\Delta)^{-1}  \one_{B((r-\frac{1}{\sqrt{2}},0), \frac{1}{\sqrt{2}})} \big) (r,0)\right]\\
&=  -\frac{\gamma}{\chi}\left(r- \frac{1}{\sqrt{2}}\right), \quad \quad \text{ for }\sqrt{2} - 1 < r \leq 1.
\end{split}
\end{equation*}
Here we used a simple explicit formula for the derivative of the inverse Laplacian applied to a radial function $f$:
\[ \partial_x (-\Delta)^{-1}f(r,0) = -\frac{1}{r} \int_0^r s f(s)\,ds. \]
Note that for $1/\sqrt{2}<r\leq 1, \,\,\,\sup_{g\in\mathcal{M}}\int g(x,y) V(x,y;r) dxdy<0$. As $\beta= v_0/\chi$, for sufficiently large $\gamma$, we can find $r_0\in (1/\sqrt{2},1)$
(due to a simple computation we may set $r_0 = \frac{v_0}{\gamma}+\frac{1}{\sqrt{2}}\leq \frac45<1$ if $\gamma$ is sufficiently large) such that
\begin{align}
  \frac{1}{\sigma}  \int g(x,y) V(x,y;r) dxdy\leq -\beta, \quad \text{ for all } g\in\mathcal{M}~\text{and}~ r_0<r\leq 1.
\end{align}
Therefore due to the structure of $\Psi$, we could set
\begin{align*}
    \d_r H(r) = -v_0,\quad \text{for}~r\in (r_0,1];
\end{align*}
instead, we use the function $\eta_2$ connecting with the last region. According to the properties prescribed in the statement of Lemma~\ref{lem418}, $\eta_2$ only makes $\partial_r H(r)$ larger.

\textbf{$\bullet$ Case 3:} $0<r\leq r_0$.  In this case, for any $g\in \mathcal{M}$, we could use a crude upper bound of \eqref{radial_comp_2d}'s right side to set
\begin{align*}
    \d_r H(r) = v_0,\quad \text{for}~ 0<r\leq r_0.
\end{align*}

Finally, we obtain \eqref{def:H_2} by integrating all cases in $r$.
\end{proof}

\section{Transport estimates}
In this section, we will prove the key estimates on the transport time it takes for the density $\rho_1$ to arrive
into the region where $\rho_2$ is supported.
Let us consider the equation dual to \eqref{eq:fokker_planck},
\begin{equation} \label{dual216}
\d_t f = \Delta f + \nabla H \cdot \nabla f,
\end{equation}
with $H(x)$ constructed in (\ref{def:H_2}). Recall that $H(x)$ is radially symmetric, with $\partial_r H(r)$ given by \eqref{defH48}. 

We will be generally interested in the solutions of \eqref{dual216} that are smooth away from the origin, non-negative, and quickly decaying at infinity.
The equation \eqref{dual216} possesses maximum principle, so the $L^\infty$ norm of $f$ is non-increasing.
It is also straightforward to check that the evolution \eqref{dual216} preserves the invariant $\int_{\R^2} f(x,t) e^{H(x)}\,dx.$

The following lemma links solutions of \eqref{eq:fokker_planck} and \eqref{dual216}, and its proof is a direct computation using integration by parts (see \cite{kiselev2022chemotaxis}).
\begin{lemma}\label{duality23}
Suppose that $u(x,t)$ and $f(x,t)$ solve \eqref{eq:fokker_planck} and \eqref{dual216} respectively
with the smooth initial data $u_0(x)= u(x,0)$ and $f_0(x) = f(x,0).$ Let $t>0,$ then for every $0 \leq s \leq t$
we have
\[ \frac{d}{ds} \int_{\R^2} f(x,s) u(x, t-s)\,dx =0. \]
In particular, it follows that
\[ \int_{\R^2} f_0(x) u(x,t)\,dx = \int_{\R^2} f(x,t) u_0(x)\,dx. \]
\end{lemma}

In the next two subsections, we will obtain lower bounds for radially symmetric functions $f$ that solve \eqref{dual216}, and then apply Lemma \ref{duality23} to obtain lower bounds of the mass of $u(\cdot,t)$ in some disk.

\subsection{Long range transport}

There are two different regimes in flux-limited chemotactic transport. The first one is from the distance $\sim L$ to the distance $\sim R_0,$ where the chemotactic drift
is saturated. In this subsection, we are concerned with this stage of the transport process.

\begin{thm}\label{trantheor}
Let $f(x,t)$ solve \eqref{dual216} with $H$ given by \eqref{def:H_2}.
Suppose that the radial initial data $f_0 \in C_0^\infty$ satisfies $\one_{B_{2R_0}}(x) \geq f_0(x) \geq \one_{B_{3R_0/2}}(x),$ and $f_0$ is non-increasing in the radial direction.
Then for all $t>0$ we have
 \begin{equation}\label{jun211}
 f(x,t) \geq \frac14 \one_{B_{R_0 + \frac18 \sqrt{4R_0^2 + \gamma t}}}(x).
 \end{equation}


\end{thm}

\begin{proof}
The solution $f(x,t)$ remains monotone non-increasing for all times - this is not hard to verify by differentiating \eqref{dual216} in radial
variable and applying comparison principle to the resulting equation (see e.g, \cite{Lieb}).

Let us first obtain a lower bound of $f(x,t)|_{\mathbb{S}_{R_0}}$. Notice that
\begin{eqnarray}
\left. f(x,t) \right|_{\Sm_{R_0}} \int_{\R^2 \setminus B_{R_0}} e^H\,dx \geq \int_{\R^2 \setminus B_{R_0}} f e^H \,dx
= \int_{\R^2} f_0 e^H \,dx - \int_{B_{R_0}} fe^H \,dx \geq \nonumber \\ \label{aux216a}
\int_{\R^2 \setminus B_{R_0}} f_0 e^H\,dx \geq \int_{\R^2 \setminus B_{R_0}} \one_{B_{3R_0/2}} e^H \,dx = \int_{B_{3R_0/2} \setminus B_{R_0}} e^H \,dx.
\end{eqnarray}

Here $\Sm_{R_0}$ is the sphere with radius $R_0$.
In the first step we used monotonicity of $f$ in the radial variable. In the second step, we use the conservation of
$\int f(x,t) e^{H(x)}\,dx$. And in the third step, we have \[ \int_{B_{R_0}} f_0 e^H \,dx  \geq \int_{B_{R_0}} f e^H\,dx \]
due to $ f_0(x) = 1 \geq f(x,t)$ in $B_{R_0}$; where we used that $f(x,t)\leq 1$ for all $x\in\mathbb{R}^2$ and $t\geq 0$ by the maximum principle. The fourth step follows from the choice of $f_0.$
Because $ H(|x|) = -\frac{\gamma}{4}\log |x|$ for $|x| \geq R_0$, we have
\begin{equation*}
    \begin{aligned}
        &\int_{\R^2 \setminus B_{R_0}} e^H\,dx  = 2\pi \int_{R_0}^{\infty} \frac{1}{r^{\frac{\gamma}{4}-1}}dr=\frac{2\pi R_0^{2-\frac{\gamma}{4}}}{\frac{\gamma}{4}-2},\\
        &\int_{B_{3R_0/2} \setminus B_{R_0}} e^H\,dx  = 2\pi \int_{R_0}^{3R_0/2} \frac{1}{r^{\frac{\gamma}{4}-1}}dr = \frac{2\pi \left(R_0^{2-\frac{\gamma}{4}}-(3R_0/2)^{2-\frac{\gamma}{4}}\right)}{\frac{\gamma}{4}-2}=\frac{2\pi R_0^{2-\frac{\gamma}{4}}(1-(3/2)^{2-\frac{\gamma}{4}})}{\frac{\gamma}{4}-2}.
    \end{aligned}
\end{equation*}
Therefore
\begin{equation}\label{flb}
    f(x,t)|_{\Sm_{R_0}} \geq 1-(3/2)^{2-\frac{\gamma}{4}}> \frac12 
\end{equation}
if $\gamma$ is large enough, so \eqref{jun211} is satisfied in $B_{R_0}$  for all $t\geq 0$ by monotonicity of $f$ in radial variable.

Next we aim to prove \eqref{jun211} in $B_{R_0}^c$, and this will be done by constructing an explicit subsolution to \eqref{dual216}. Consider a smooth, convex function $\omega(r)$ defined on $[R_0, \infty)$ such that $\omega(R_0)=1/2,$ $\omega(r)\geq 1/4$ for $r \in [R_0,5R_0/4],$ and $\omega(r)$ vanishes if $r \geq 3R_0/2.$
Note that these assumptions imply that $\omega'(r) < 0$ for $r \in [R_0,3R_0/2].$
Let us define $\omega_{\varphi(t)}(r) = \omega(R_0 +\varphi(t)(r-R_0)).$ Here $\varphi(t)$ will be a decreasing function satisfying $\varphi(0)=1$ and $\varphi(t) \geq 0,$ that will be designed so that $\omega_{\varphi(t)}(r)$ can serve as a barrier for $f(x,t).$

Note that
\begin{equation}\label{aux47a} \partial_t \omega_{\varphi(t)}(r) = \omega'(R_0+\varphi(t)(r-R_0)) (r-R_0)\varphi'(t), \end{equation}
 while
\begin{align}
(\Delta + \nabla H \cdot \nabla )\omega_{\varphi(t)}(r) = \partial_r^2 \omega_{\varphi(t)}(r) + \frac{1}{r}\partial_r\omega_{\varphi(t)}(r) + \partial_r H(r) \partial_r \omega_{\varphi(t)}(r) \nonumber \\
\geq \frac{\gamma -4}{4r}|\partial_r  \omega_{\varphi(t)}(r)| \geq \frac{\gamma}{8r} |\omega'(R_0+\varphi(t)(r-R_0))| \varphi(t). \label{aux47b}
\end{align}
Here we used convexity of $\omega,$ \eqref{defH48}, and the fact that $\gamma$ is sufficiently large.

Due to assumptions on $\omega(r)$ and $f_0(x),$ and \eqref{flb}, we will have $f(x,t) \geq \omega_{\varphi(t)}(r)$ for all $x \in [R_0,\infty)$ and $t \geq 0$ provided that
\[   \partial_t \omega_{\varphi(t)}(r) \leq (\Delta + \nabla H \cdot \nabla )\omega_{\varphi(t)}(r) \]
in this region.
This translates into the inequality
\[ -\varphi'(t)(r-R_0) \leq \frac{\gamma}{8r} \varphi(t) \]
that needs to hold for all $r \geq R_0$ and $t \geq 0$ such that $R_0+\varphi(t)(r-R_0) \in [R_0, 3R_0/2]$ (since otherwise the factor $\omega'(R_0+\varphi(t)(r-R_0))$ present in \eqref{aux47a}, \eqref{aux47b} vanishes).
Hence at every time $t \geq 0$ we should have
\[ -\frac{\varphi'(t) (r-R_0)r}{\varphi(t)} \leq \frac{\gamma}{8} \]
for all $r \in [R_0, R_0 + \frac{R_0}{2\varphi(t)}].$ The inequality is most restrictive when $r =R_0+ \frac{R_0}{2\varphi(t)},$
leading to condition
\[ -\varphi'(t) R_0^2 \left(\frac{1}{2\varphi(t)^2} + \frac{1}{4\varphi(t)^3} \right) \leq \frac{\gamma}{8}. \]
Since $\varphi(t) \leq 1$ and $\varphi'(t) \leq 0,$ this inequality is satisfied if
\[ -\frac{\varphi'(t)}{\varphi(t)^3} \leq \frac{\gamma}{8R_0^2}. \]
Integrating in time and using $\varphi(0)=1$, we get
\[ \frac{1}{\varphi(t)^2} \leq 1 + \frac{\gamma t}{4 R_0^2} \]
and finally
\begin{equation}\label{philb}
\varphi(t) \geq \frac{2R_0}{\sqrt{4R_0^2 + \gamma t}}.
\end{equation}
Due to the assumptions on $\omega,$ we have that
\[\omega_{\varphi(t)}(r) = \omega(R_0 +\varphi(t)(r-R_0)) \geq \frac14 \,\,\,\,{\rm  if}\,\,\,\, R_0 + \varphi(t)(r-R_0) \leq \frac{5R_0}{4}.\]
Given \eqref{philb}, this yields $r \leq R_0 + \frac{1}{8}\sqrt{4R_0^2+\gamma t}.$ As we discussed, $f(x,t) \geq \omega_{\varphi(t)}(|x|);$ hence we arrive at \eqref{jun211}.
\end{proof}

Now we can use the result of Theorem~\ref{trantheor} and the duality Lemma~\ref{duality23} to arrive at the following
\begin{cor}\label{longrangecor}
Let $u(x,t)$ solve \eqref{eq:fokker_planck} with a potential $H$ given by \eqref{def:H_2}.
Suppose that the initial data $u_0\geq 0$ satisfies $\int_{0\leq |x| \leq L} u_0(x)\,dx \geq  M_0$, with $L \geq R_0$.
Then for all $t\geq \frac{64(L-R_0)^2}{\gamma}$, we have
\begin{equation}\label{aux217a}
\int_{B_{2R_0}} u(x,t) \, dx \geq  \frac{M_0}{4}.
\end{equation}
\end{cor}
\begin{remark}
In fact, the estimate for $t$ can be improved, but only by $\sim R_0^2/\gamma;$ this improved would anyway by absorbed by the time necessary to pass through the next stage of transport process.
\end{remark}
\begin{proof}
We choose a radial initial data $f_0 \in C_0^\infty$ as in the previous theorem: $f_0$ is non-increasing in the radial direction and satisfies  $\one_{B_{2R_0}}(x) \geq f_0(x) \geq \one_{B_{3R_0/2}}(x)$.
Then
\[  \int_{B_{2R_0}} u(x,t) \, dx \geq  \int_{\R^2} f_0(x) u(x,t)\,dx  = \int_{\R^2} f(x,t) u_0(x)\,dx\]
by the duality.
A simple computation using \eqref{jun211} shows that if $t\geq \frac{64(L-R_0)^2}{\gamma}$, then $f(x,t) \geq \frac14 \one_{B_L}(x).$
This completes the proof.

\end{proof}



\subsection{Transport within $|x| \sim R_0$ region}\label{R0transport}

The second stage of chemotactic transport is from the distance $\sim R_0$ to the distance $\sim v_0^{-1}.$ At the latter distance the strength of the drift and diffusion balance,
and much of $\rho_1$ mass will remain at this distance from the origin. Indeed, if we consider the stationary Fokker-Planck operator
$Af =-\Delta f+ \nabla \cdot (f \nabla H)$ with potential $H = -v_0 |x|,$ it is not difficult to verify that it has a ground state $e^H = e^{-v_0r}$ that is
concentrated exactly on the scale $\sim v_0^{-1}:$ balls with radii of smaller order will only contain a correspondingly small fraction of the total mass of the ground state.

Recall that $M_u(r,t)= \int_{B(0,r)} u(x,t) dx$, and it satisfies
\begin{equation}\label{eqn:dec81}
\begin{aligned}
\partial_t M_u &= \partial_{rr} M_u - \frac{1}{r} \partial_r M_u - \partial_r M_u ~\partial_r H\\
& = \partial_{rr} M_u - \Big(\frac{1}{r} +\partial_r H\Big)\partial_r M_u
\end{aligned}
\end{equation}
We proved in the previous subsection that when $t\geq \frac{64(L-R_0)^2}{\gamma}$, we have $M_u(2R_0, t)\geq \frac{M_0}{4}$.
What we will do next is investigate the mass transport for $t\geq t_0=\frac{64(L-R_0)^2}{\gamma}$, by restarting the process under the assumption
\begin{equation}
    \int_{B_{2R_0}} u(x,t_0) dx \geq \frac{M_0}{4}.
\end{equation}
\begin{thm}\label{thm:mass_v0}
Let $u(x,t)$ satisfy \eqref{eq:fokker_planck} with $H$ given by \eqref{def:H_2}.
    Suppose that at a given time $t_0$,
    \begin{equation}\label{MassumpR0}
    \int_{B_{2R_0}} u(x,t_0) dx \geq c_0 M_0,
\end{equation}
with $R_0 = \farc{\gamma}{2v_0}-1$. Then there exists a universal constant $C_1>0$ such that
\begin{equation}
    \int_{B_{5/v_0}} u(x,t) dx \geq \frac{c_0 M_0}{5},\quad \text{for}~~t\geq t_0 + \frac{C_1\gamma}{v_0^2}.
\end{equation}
\end{thm}

\begin{proof}
Let $d_0 = \frac{2}{v_0}$ and $d_1 = \frac{4}{v_0}$. The argument is similar to the previous subsection and proceeds by looking at the dual equation (\ref{dual216}).
Let the initial data be a radial function $f_0\in C_0^{\infty}(B_{5/v_0})$ that is non-increasing in the radial direction, and  satisfies $1 \geq f_0(x)\geq \one_{B_{d_1}}(x)$.

Then we have  the following from a similar argument as \eqref{aux216a}: 
\begin{eqnarray}
\left. f(x,t) \right|_{\Sm_{d_0}} \int_{\R^2 \setminus B_{d_0}} e^H\,dx \geq \int_{\R^2 \setminus B_{d_0}} f e^H \,dx
= \int_{\R^2} f_0 e^H \,dx - \int_{B_{d_0}} fe^H \,dx \geq \nonumber \\ \label{aux216b}
\int_{\R^2 \setminus B_{d_0}} f_0 e^H\,dx \geq \int_{\R^2 \setminus B_{d_0}} \one_{B_{d_1}} e^H = \int_{B_{d_1} \setminus B_{d_0}} e^H \,dx.
\end{eqnarray}
Since we assume that $v_0 \leq 1,$ we have $d_0 > 1$.  Observe that $d_1/d_0=2$, and set $K_0:=\inf\{k\in \mathbb{N}:2^{k}d_0\geq R_0-1\}$.
Using the definition of $H$ in (\ref{def:H_2}), we estimate

\begin{equation}\label{aux48a}
    \begin{aligned}
        \int_{\R^2 \setminus B_{d_0}} e^H\,dx &=  \int_{\R^2 \setminus B_{R_0}} e^H\,dx+ \int_{B_{R_0}\setminus B_{d_0}} e^H\,dx\\ 
     &    = 2\pi \int_{R_0}^{\infty} \frac{1}{r^{\frac{\gamma}{4}-1}}dr 
         +  2\pi \int_{R_0-1}^{R_0} e^{-\int_r^{R_0} \eta_1(s)\,ds -\frac{\gamma \log R_0}{4}}r\,dr\\ & \quad + 2\pi \int_{d_0}^{R_0-1} e^{-v_0 r+v_0(R_0-1)-\int_{R_0-1}^{R_0} \eta_1(s)\,ds - \frac{\gamma \log R_0}{4}} r dr \\
         &\leq \frac{2\pi R_0^{2-\frac{\gamma}{4}}}{\frac{\gamma}{4}-2} + 2\pi R_0^{1-\frac{\gamma}{4}} e^{v_0} + 2\pi \sum_{k=0}^{K_0} \int_{2^k d_0}^{2^{k+1}d_0}e^{-v_0 r+v_0(R_0-1)-\int_{R_0-1}^{R_0} \eta_1(s)\,ds - \frac{\gamma \log R_0}{4}} r dr  \\
       & \leq  \frac{2\pi R_0^{2-\frac{\gamma}{4}}}{\frac{\gamma}{4}-2} + 2\pi R_0^{1-\frac{\gamma}{4}} e^{v_0} + 2\pi \sum_{k=0}^{K_0} e^{(1-2^k)v_0d_0} 2^{2k}\int_{d_0}^{d_1}e^{-v_0 r+v_0(R_0-1)-\int_{R_0-1}^{R_0} \eta_1(s)\,ds
         - \frac{\gamma \log R_0}{4}} r dr \\
         & \leq  \frac{2\pi R_0^{2-\frac{\gamma}{4}}}{\frac{\gamma}{4}-2} +  2\pi R_0^{1-\frac{\gamma}{4}} e^{v_0} + C_2 \int_{B_{d_1} \setminus B_{d_0}} e^H \,dx,
    \end{aligned}
\end{equation}
where $C_2$ is a universal constant that, since $v_0d_0=2,$ is equal to $\sum_{k=0}^\infty 2^{2k} e^{2(1-2^k)};$ it is not hard to estimate that $C_2 \leq 1.75.$
On the other hand,
\begin{equation}
\begin{split}
\int_{B_{d_1} \setminus B_{d_0}} e^H \,dx &= 2\pi \int_{d_0}^{d_1}e^{-v_0 r+v_0(R_0-1) -\int_{R_0-1}^{R_0} \eta_1(s)\,ds-\frac{\gamma \log R_0}{4}} r dr \\ & \geq 2\pi \int_{d_0}^{2d_0} e^{-v_0 r} rdr e^{v_0 (R_0-\frac12)} R_0^{-\frac{\gamma}{4}} \\
& \geq
\frac{c_1}{v_0^2}  e^{v_0 R_0} R_0^{-\frac{\gamma}{4}} = \frac{c_1}{v_0^2}  e^{-v_0 +\frac{\gamma}{2}} R_0^{-\frac{\gamma}{4}}, \label{aux48b}
\end{split}
\end{equation}
where $c_1>0$ is a universal constant. Here we used the identity $v_0 d_0 =2$, the fact that $-\frac{v_0}{2} \geq \eta_1(r),$ and the definition $R_0 = \frac{\gamma}{2v_0} -1$. From \eqref{aux48a}, \eqref{aux48b} and our assumption that $v_0 \leq 1$
it is clear that for sufficiently large $\gamma,$ we have
\begin{equation}\label{aux48c}
\left. f(x,t) \right|_{\Sm_{d_0}} \geq  \int_{B_{d_1} \setminus B_{d_0}} e^H \,dx \left( \int_{\R^2 \setminus B_{d_0}} e^H\,dx \right)^{-1} \geq \frac12.
\end{equation}

Now fix a convex smooth function $\omega$ on $[d_0,\infty)$ such that $\omega(d_0)=\frac12$, $\frac12 \geq \omega(r)>0$ for $r \in [d_0,d_1),$ $\omega(r) \geq \frac15$ if $r \in [d_0,\frac32 d_0],$
and $\omega(r)=0$ if $r \geq d_1.$ Observe that our assumptions imply that $\omega(r)$ is monotone decreasing in $r.$
Define $\omega_\varphi(r) = \omega (d_0 + \varphi(t)(r-d_0))$ for $r \geq d_0$ and $t \geq t_0.$ 
The function $\varphi(t)$ satisfies $1 \geq \varphi(t) \geq 0,$ $\varphi(t_0)=1,$ and is monotone decreasing.
Note that, similarly to the previous subsection,
\begin{align}\nonumber
(\Delta + \nabla H \cdot \nabla) \omega_\varphi(r) & = \partial_r^2 \omega_\varphi (r) + \frac1r  \partial_r \omega_\varphi(r) + \partial_r H(r) \partial_r \omega_\varphi(r)
\geq
\left(|\partial_r H(r)| - \frac{1}{r}\right)| \omega'(d_0+\varphi(t)(r-d_0))|\varphi(t) \\ & \geq \frac{v_0}{16}| \omega'(d_0+\varphi(t)(r-d_0))|\varphi(t),\quad \text{if}~~d_0<r\leq 4R_0, \label{aux48d}
\end{align}
provided that $\gamma$ is sufficiently large. Here we used \eqref{defH48} and the fact that $d_0>1$ in estimating $\partial_r H.$ Indeed, in the $R_0-1 \geq r \geq d_0$ range
\[ \left||\partial_r H(r)| - \frac1r \right| \geq v_0 - \frac{1}{d_0} = \frac{v_0}{2}.\]
In the $4R_0 \geq r \geq R_0-1$ range, we have
\[ \left||\partial_r H(r)| - \frac1r \right| \geq \frac{\gamma-4}{4r} \geq \frac{v_0(\gamma - 4)}{8\gamma} \geq \frac{v_0}{16} \]
if $\gamma$ is large enough.
In addition,
\begin{align}\label{aux48e}
\partial_t  \omega_\varphi(r) = (r-d_0) \omega'(d_0+\varphi(t)(r-d_0))\varphi'(t)
\end{align}
and $\partial_t \omega_{\varphi(t)} = (\Delta + \nabla H \cdot \nabla) \omega_{\varphi(t)}=0$ if $d_0 +\varphi(t)(r-d_0) \geq d_1$.

We will choose a decreasing $\varphi(t)$ so that $\omega_{\varphi(t)}(r)$ is a subsolution of the equation \eqref{dual216}, satisfying
\[ \partial_t \omega_{\varphi(t)} \leq (\Delta +\nabla H \cdot \nabla) \omega_{\varphi(t)}. \]
Also, the function $\varphi(t)$ will be chosen so that $d_0 +\varphi(t)(4R_0 -d_0) \geq d_1$ for all $t \geq t_0.$
This will ensure that we can use the bound \eqref{aux48d} for all $t \geq t_0:$ it holds for $r \leq 4R_0,$ while all larger $r> 4R_0$ lie in the range
where the subsolution $\omega_{\varphi(t)}$ vanishes trivially.

If we are able to select $\varphi(t)$ that satisfies all the above properties, then due to the assumptions on $f_0(x)$ that imply $f_0(x) \geq \omega(r)$ and the boundary control $\left. f(x,t) \right|_{\Sm_{d_0}} \geq \frac12 = \left. \omega_{\varphi(t)}(r) \right|_{\Sm_{d_0}},$
we will have $f(x,t) \geq \omega_{\varphi(t)}(r)$ on $\R^2 \setminus B_{d_0}$ for all times $t \geq t_0.$


For $d_0< r\leq 4R_0$, we need to ensure that
\[ -(r-d_0)\varphi'(t) \leq \frac{v_0}{16} \varphi. \]
Since only the region $r-d_0\leq \frac{d_1-d_0}{\varphi(t)}$ plays any role and $\varphi'(t)\leq 0$, it suffices to require that
\begin{equation*}
    (d_1-d_0)\d_t\big(1/\varphi(t)\big)\leq \frac{v_0}{16}.
\end{equation*}
Then we can choose
\begin{align}
    \varphi(t) = \frac{1}{1+\frac{v_0}{16(d_1-d_0)} (t-t_0)}, \quad \text{for}~0\leq t\leq t_*,
\end{align}
 where $t_*>0$ is the time when the support of the subsolution reaches $r = 4R_0.$ Namely, $t_*$ is determined by
 \begin{equation}
     d_0 +\varphi(t_*)(4R_0-d_0)=d_1,
 \end{equation}
 so that
 \begin{equation}\label{jun162}
     \varphi(t_*) = \frac{d_1-d_0}{4R_0-d_0},\quad \text{and}\quad t_*=\frac{16}{v_0}(4R_0-d_1) +t_0 \leq \frac{C_1 \gamma}{v_0^2} +t_0;
 \end{equation}
 the latest inequality holds if $\gamma$ is sufficiently large.
 When $t\geq t_*$, we just define $\varphi(t)\equiv \varphi(t_*) = \frac{d_1-d_0}{4R_0-d_0}$.
 Such definition of $\varphi(t)$ satisfies all the properties required for $\omega_{\varphi(t)}$ to be a subsolution,
 and so $f(x,t) \geq \omega_{\varphi(t)}(r)$ for all $|x| \geq d_0$ and $t \geq t_0.$
 Therefore, we conclude that
 \[   f(x,t) \geq \omega\left(d_0 +\frac{d_1-d_0}{4R_0-d_0}(r-d_0)\right) \] 
for all $t \geq t_*.$ In particular, at $r = 2R_0,$ we have
\[ 
d_0 +\frac{d_1-d_0}{4R_0-d_0}(2R_0-d_0) \leq d_0 + \frac12 (d_1-d_0) = \frac32 d_0. \]
Due to our assumptions on $\omega,$ we conclude that
\begin{equation}\label{aux49a} f(x,t) \geq f(2R_0,t) \geq \omega(3d_0/2) \geq \frac15 \end{equation}
for $|x| \leq 2R_0$ 
when $t \geq t_*;$ we used monotonicity of $f(x,t)$ in the first inequality.

Recall that $f_0 \in  C_0^{\infty}(B_{5/v_0})$ and $f_0(x) =1$ for $x \in B_{d_1}.$
Then by duality, for every $t \geq t_0,$
\[  \int_{B_{5/v_0}} u(x,t) \, dx \geq  \int_{\R^2} f_0(x) u(x,t)\,dx  = \int_{\R^2} f(x,t) u(x,t_0)\,dx \geq \frac{c_0 M_0}{5} \]
due to \eqref{aux49a} and \eqref{MassumpR0}.
\end{proof}

\section{Harnack's inequality and reaction in $B_1$}

By Theorem \ref{thm:mass_v0} and (\ref{ineq:rho_vs_u}), we have
\be\label{apr91}
\int_{B(0,\frac{5}{v_0})} \rho_1(x,t) dx \geq \frac{1}{20} M_0-\pi\theta \geq c_2 M_0,
\ee
for some $c_2 >0$ for all $t\in[t_*, \tau]$ with $t_*$ given by \eqref{jun162}, which satisfies the following using the definition $t_0=\frac{64(L-R_0)^2}{\gamma}$:
\be \label{aux411b} t_* \leq C\Big( \frac{(L-R_0)^2}{\gamma} + \frac{\gamma}{v_0^2}\Big). \ee 
The last inequality in \eqref{apr91} follows from the assumption $M_0 \gg \theta;$ 
specifically, it suffices to assume that $M_0 \geq 40\pi \theta,$ then one can take $c_2 = \frac{1}{40}.$


Therefore we know that a large mass of $\rho_1$ will be inside the ball of radius $\sim v_0^{-1}$ by the time $t_*$ at latest. We now need to show that
some of this mass will also appear in $B_1$ and react with the density $\rho_2.$ To prove this, we will prove a variant of Harnack inequality showing that the mass has to be spread
relatively uniformly over $B_{\frac{5}{v_0}}$ for times $t$ such that $\tau \geq t \geq t_*+ v_0^{-2}.$

Let us first look at the following equation without reaction term:
\begin{align}
    \d_t \rho -\Delta \rho +\nabla \cdot(b \rho) = 0,
\end{align}
where $\norm{b(\cdot,t)}_{L^{\infty}(\R^n)}\leq v_0$ for all $t\geq 0$.
The corresponding heat kernel $\Gamma(x,y,t,s)$, according to \cite{aronson1968non}, satisfies $\Gamma(x,y,t,s)=\tilde \Gamma(y,x,t,s)$, where $\tilde \Gamma$ is the heat kernel for the adjoint equation
\begin{align}\label{dd411}
    \d_t \phi -\Delta \phi - b \cdot\nabla \phi = 0.
\end{align}
We will use the following estimates of the fundamental solution to drift-diffusion equations in $n$ dimensions with uniformly bounded drift.
\begin{lemma}[\cite{hill1997estimates}, Theorem 2.1; \cite{hamel2020propagation}, Lemma 4.2] Given a bounded vector field $b(x,t)$ such that $\norm{b}_{L^{\infty}((0,+\infty)\times \R^n)}\leq B<\infty$, let $\Gamma(x,y,t,s)$ be the fundamental solution satisfying
\begin{equation}
    \begin{aligned}
        &\Gamma_t + b \cdot \nabla_x \Gamma = \Delta_x \Gamma\\
        & \Gamma(x,y,s,s) = \delta_y(x)
    \end{aligned}
\end{equation}
    for any $t>s\geq0$ and $y\in\R^n$. Then we have
        \begin{equation}\label{lowerbd}
    \begin{aligned}
        \Gamma(x,y,t,s)&\geq   \prod_{i=1}^n \Big(\frac{1}{\sqrt{4\pi (t-s)}}e^{-\frac{(|x_i-y_i|+B (t-s))^2}{4(t-s)}}-\frac{B}{4}\erfc\Big(\frac{|x_i-y_i|}{\sqrt{4(t-s)}}+\frac{B\sqrt{t-s}}{2}\Big)\Big).
        \end{aligned}
    \end{equation}
\end{lemma}
A corresponding upper bound on $\Gamma$ can also be found in \cite{hill1997estimates,hamel2020propagation}, but we will not need it for our application.
Recall that
\[ \erfc(x) = \frac{2}{\sqrt{\pi}} \int_x^\infty e^{-y^2}\,dy. \]
Note that in our case, the drift is of the form $b= \frac{\nabla c}{|\nabla c|}\Psi(|\nabla c|)$, and thus we may take $B=v_0$ in the bound (\ref{lowerbd}) to obtain the following estimates.
\begin{lemma}\label{mar601}
 Consider $n=2$. Suppose that the drift in the equation \eqref{dd411} satisfies $\|b\|_{L^\infty} \leq v_0.$
    If $ t = v_0^{-2}$ and $|x-y| \leq C_3 v_0^{-1},$ then the fundamental solution satisfies
    \begin{equation}\label{lowerG}
        \Gamma(x,y,t,0)\geq a v_0^2,
    \end{equation}
    where $a$ depends only on $C_3.$
\end{lemma}
\begin{proof}
To prove \eqref{lowerG}, we will use \eqref{lowerbd}. Set $B=v_0,$ and let us introduce shortcut notation $d = |x_i-y_i|.$ Note that
\[ \frac{v_0}{4}\erfc\left(\frac{d}{\sqrt{4t}}+\frac{v_0\sqrt{t}}{2}\right) = \frac{v_0}{\sqrt{4\pi}} \int^\infty_{\frac{d+v_0 t}{\sqrt{4t}}}e^{-y^2}\,dy. \]
Also
\be\label{aux419a} \frac{v_0}{\sqrt{4\pi}} \int^\infty_{\frac{d+v_0 t}{\sqrt{4t}}}e^{-y^2}\,dy \leq  \frac{\sqrt{t} v_0}{\sqrt{\pi}(d+v_0 t)} \int^\infty_{\frac{d+v_0 t}{\sqrt{4t}}}e^{-y^2}y\,dy. \ee
Integrating the last integral and combining it with the first term in \eqref{lowerbd} that is equal to $\frac{1}{\sqrt{4\pi t}}e^{-\frac{(d-v_0 t)^2}{4t}},$
we get
\[ \frac{1}{\sqrt{4\pi t}} \left(1- \frac{v_0 t}{d+v_0 t} \right) e^{-\frac{(d-v_0 t)^2}{4t}} \geq 0. \]
Let us introduce the short-cut notation $\frac{d+v_0 t}{\sqrt{4t}} = \xi.$
On the other hand, another non-negative contribution to each factor in \eqref{lowerbd} comes from the difference of the right and left sides in \eqref{aux419a}.
It equals
\be\label{aux419b} \frac{v_0}{\sqrt{4\pi}} \int_\xi^\infty \left(\frac{y}{\xi}-1 \right) e^{-y^2}\,dy = \frac{v_0 \xi}{\sqrt{4\pi}} \int_1^\infty (z-1) e^{-\xi^2 z^2}\,dz. \ee
Observe that under assumptions of the lemma, we have
\[ \xi = \frac{d+v_0 t}{\sqrt{4t}} = \frac{d +v_0^{-1}}{2v_0^{-1}} \in [\frac12, \frac12 +C_3]. \]
Substituting this into \eqref{aux419b} leads to the lower bound $\gtrsim v_0.$
Given that there are two factors when $n=2,$ the estimate \eqref{lowerG} follows.
\end{proof}
We use the lower bound \eqref{lowerG} on $\Gamma$ to show that the density $\rho_1$ will spread over $B_1$ in a fairly uniform fashion, at least at times exceeding $t_*+v_0^{-2}.$
The last obstacle is handling the effects of reaction.

\begin{lemma}
Recall our assumption that $v_0^2 M_0 \gg \theta.$
Consider the equation
\begin{align}
    \d_t \rho_1 - \Delta \rho_1 +\nabla\cdot\Big(\rho_1 \frac{\nabla c}{|\nabla c|}\Psi(|\nabla c|)\Big) = -\epsilon \rho_1 \rho_2,
\end{align}
and suppose
\begin{equation}\label{aux411a} \int_{B_{\frac{5}{v_0}}} \rho_1(x,t) \,dx \geq c_2 M_0 \end{equation}
for all $t$ such that $\tau \geq t \geq t_*.$ Let $t$ be any time satisfying $\tau \geq t \geq t_* + v_0^{-2}.$ There exists 
a set $S(t) \subset B_1$ with
$|S(t)| \leq \frac{4 \pi \theta}{ac_2 M_0 v_0^2}$ such that for every $x \in B_1 \setminus S_t$ we have
\begin{equation}\label{rho1_lowerbd}
    \rho_1(x,t)\geq ac_2 v_0^2 M_0.
\end{equation}
Here $a$ is the constant from \eqref{lowerG} corresponding to $C_3 = 5.$
\end{lemma}
\begin{proof}
By the Duhamel's principle, we have that for any $t$ satisfying $\tau \geq t \geq t_*+v_0^{-2}$ and
for every $x\in B_1$,
\begin{equation}
\begin{aligned}
    \rho_1(x,t)&= \int_{\R^2} \Gamma(x,y, t, t- v_0^{-2}) \rho_1(y,t-v_0^{-2})dy -\epsilon \int_{t-v_0^{-2}}^t\int_{\R^2}\Gamma(x,y,t,s) \rho_1(y,s) \rho_2(y,s) dy ds\\
    &\geq \int_{B_{\frac{5}{v_0}}} \Gamma(x,y, t, t-v_0^{-2}) \rho_1(y,t-v_0^{-2})dy -\epsilon \int_{0}^t\int_{\R^2}\Gamma(x,y,t,s) \rho_1(y,s) \rho_2(y,s) dy ds\\
    &\geq ac_2 v_0^2 M_0 - g(x,t).
\end{aligned}
\end{equation}
Here in the second step we used the assumption \eqref{aux411a} (which we know holds by \eqref{apr91}), and the lower bound \eqref{lowerG}. We also defined
 $g(x,t):= \epsilon \int_{0}^t\int_{\R^2}\Gamma(x,y,t,s) \rho_1(y,s) \rho_2(y,s) dy ds$. Observe that
\begin{equation}
        \int_{\Rm^2} g(x,t) dx = \epsilon \int_0^t\int_{\Rm^2}\rho_1(y,s)\rho_2(y,s)dy ds = \int_{\Rm^2} (\rho_2(x,0)-\rho_2(x,t)) dx
        \leq\frac{1}{2} \|\rho_2(\cdot, 0)\|_1 \leq \pi \theta,
\end{equation}
as $t\leq \tau$.
Therefore, $g(x,t)$ can exceed $\frac12 ac_2 v_0^2 M_0$ only on a set $S(t)$ with $|S(t)| \leq \frac{2\pi \theta}{ac_2 v_0^2 M_0}.$
For every $x \in B_1 \setminus S(t),$ we must have $\rho_1(x,t) \geq  \frac12 ac_2 v_0^2 M_0.$
\end{proof}

We are ready to complete the proof of the main Theorem \ref{thm:main}:

\begin{proof}[Proof of Theorem \ref{thm:main}]
Consider any $t\in (t_*+ v_0^{-2}, \tau].$ 
We have
\[     \frac{d}{dt}\int_{B_1} \rho_2(x,t) dx = -\epsilon\int_{B_1} \rho_1(x,t) \rho_2(x,t)dx\leq -\epsilon c_3v_0^2M_0 \int_{B_1\setminus S(t)} \rho_2(x,t)dx, \]
where we set $c_2 = \frac12 ac_2.$
For all $t \leq \tau,$ we have $\int_{B_1} \rho_2(x,t)\,dx \geq \pi \theta.$ In addition, $\rho_2(x,t) \leq 2\theta$ for all $x,t.$ Due to estimate on the measure of $S(t),$
we have that \[ \int_{S(t)} \rho_2(x,t)\,dx \leq \frac{8 \pi \theta^2}{ac_2 M_0 v_0^2} \leq \frac12 \pi \theta \]
since $v_0^2 M_0 \gg \theta.$ Therefore,
\[ \int_{B_1 \setminus S(t)} \rho_2(x,t) \, dx \geq \frac12 \pi \theta \geq \frac14 \int_{B_1} \rho_2(x,t)\,dx. \]
Hence
\[  \frac{d}{dt}\int_{B_1} \rho_2(x,t)\, dx \leq -\epsilon \frac14 c_3v_0^2M_0 \int_{B_1} \rho_2(x,t)\, dx. \]
We conclude that there exists a universal constant $C$ such that
$\tau \leq t_* + v_0^{-2} + \frac{C}{\epsilon v_0^2 M_0}.$ Given the expression for $t_*$ in \eqref{aux411b} and since we assume that $\gamma \gg 1,$
we finally arrive at the bound \[ \tau \leq C \left( \frac{L^2}{\gamma} + \frac{\gamma}{v_0^2}  + \frac{1}{\epsilon v_0^2 M_0}\right). \]
\end{proof}


\section*{Acknowledgment}

AK has been partially supported by the NSF-DMS award 2306726. YY has been supported by the NUS startup grant, MOE Tier 1 grant, and the Asian Young Scientist Fellowship.


\begin{thebibliography}{99}



\bibitem{aronson1968non} D. G. Aronson, \it Non-negative solutions of linear parabolic equations, \rm  Ann. Scuola Norm. Sup. Pisa Cl. Sci. (3) {\bf 22} (1968), 607--694

\bibitem{BBTW} N.~Bellomo, A.~Bellouquid, Y.~Tao, and M.~Winkler, \it Toward a mathematical theory of Keller-Segel
models of pattern formation in biological tissues, \rm Math. Model Methods Appl. Sci. {\bf 25} (2015), 1663--1763

\bibitem{BW} N.~Bellomo and M.~Winkler, \it A degenerate chemotaxis system with flux limitation: Maximally
extended solutions and absence of gradient blow-up, \rm Commun. Part. Diff. Eq. {\bf 42} (2017), 436--473

\bibitem{CPY} V.~Calvez, B.~Perthame, and S.~Yasuda, \it Traveling wave and aggregation in a flux-limited Keller-Segel model,
\rm  Kinetic and Related Models, AIMS {\bf 11} (2018), 891--909

\bibitem{CKWW} A.~Chertock, A.~Kurganov, X.~Wang, and Y.~Wu, On a chemotaxis model with saturated
chemotactic flux, \rm Kinetic and Related Models. {\bf 5} (2012), 51--95

\bibitem{ccw} J. P. Crimaldi, J. R. Cadwell, and J. B. Weiss, \it Reaction enhancement by isolated scalars by vortex stirring, \rm Physics of Fluids, {\bf 20} (2008), 073605

\bibitem{chw} J. P. Crimaldi, J. R. Hartford, and J. B. Weiss, \it  Reaction enhancement of point sources due to vortex stirring, \rm Phys. Rev. E, {\bf 74} (2006), 016307

\bibitem{ds} M. W. Denny and M. F. Shibata, \it Consequences of surf-zone turbulence for settlement and external fertilization, \rm Am. Nat., {\bf 134} (1989), 859--889

\bibitem{Desh}
S.~Deshmane, S.~Kremlev, S.~Amini, and B.~Sawaya, \it  Monocyte Chemoattractant Protein-1 (MCP-1): An Overview,
\rm J Interferon Cytokine Res. {\bf 29} (2009), no. 6, 313--326

\bibitem{DS} Y. Dolak and C. Schmeiser, \it Kinetic models for chemotaxis: Hydrodynamic limits and spatiotemporal
mechanisms, \rm J. Math. Biol. {\bf 51} (2005), 595

\bibitem{chemoshear}  Y. Gong, S. He and A. Kiselev, \it Random search in shear flow subject to chemotaxis, \rm Bulletin of Math. Biology  {\bf 84} (2022), no. 7, Paper No. 71, 46 pp

\bibitem{hamel2020propagation} F. Hamel and C. Henderson, \it Propagation in a Fisher-KPP equation with non-local advection, \rm  J. Funct. Anal. {\bf 278} (2020), no. 7, 108426, 53 pp

\bibitem{hill1997estimates} A. T. Hill, \it Estimates on the heat kernel of parabolic equations with advection, \rm SIAM J. Math. Anal. {\bf 28} (1997), no. 6, 1309--1316

\bibitem{HP} T.~Hillen and K.J.~Painter, \it A user's guide to PDE models for chemotaxis, \rm J. Math. Biol. {\bf 58} (2009), no. 1-2, 183--217

\bibitem{HPS} T.~Hillen, K.~Painter, and C.~Schmeiser, \it Global existence for chemotaxis with finite sampling
radius, \rm Discrete Contin. Dyn. Syst. Ser. B  {\bf 7} (2007), 125--144

\bibitem{HRZZ} J.E.~Himes, J.A.~Riffel, C.A.~Zimmer and R.K.~Zimmer, \it Sperm chemotaxis as revealed with live and
synthetic eggs, \rm Biol. Bull. {\bf 220} (2011), 1--5

\bibitem{JV1} F.~James and N.~Vauchelet, \it Chemotaxis: from kinetic equations to aggregate dynamics,
\rm Nonlinear Differ. Equ. Appl. {\bf 20} (2013), 101

\bibitem{KS1} E.F.~Keller and L.A.~Segel, \it Initiation of slide mold aggregation
viewed as an instability, \rm J. Theor. Biol. {\bf 26} (1970), 399--415

\bibitem{KS2} E.F.~Keller and L.A.~Segel, \it Model for chemotaxis,
\rm J. Theor. Biol. {\bf 30} (1971), 225--234

\bibitem{KR1} A. Kiselev and L. Ryzhik, \it Biomixing by chemotaxis and enhancement of biological reactions, \rm Comm. PDE, {\bf 37} (2012), 298--318

\bibitem{KR2} A. Kiselev and L. Ryzhik, \it Biomixing by chemotaxis and efficiency of biological reactions: the critical reaction case,
\rm J. Math. Phys. {\bf 53} (2012), no. 11, 115609

\bibitem{kiselev2022chemotaxis} A. Kiselev, F. Nazarov, L. Ryzhik and Y. Yao, \it Chemotaxis and reactions in biology, \rm J. Eur. Math. Soc. (JEMS) {\bf 25} (2023), no. 7, 2641--2696

\bibitem{Lasker} H. Lasker, \it High fertilization success in a surface-brooding Carribean Gorgonian, \rm Biol. Bull. {\bf 210} (2006),
10--17

\bibitem{Lieb} G.M. Lieberman, \it Second order parabolic differential equations, \rm World Scientific Publishing Co., Inc., River Edge, NJ, 1996


\bibitem{Patlak} C.S.~Patlak, \it Random walk with perisistence and external bias, \rm Bull. Math. Biol. Biophys. {\bf 15} (1953), 311--338

\bibitem{Penn} J. Pennington, \it The ecology of fertilization of echinoid eggs: The consequences of sperm dilution, adult
aggregation and synchronous spawning, \rm Biol. Bull. {\bf 169} (1985), 417--430

\bibitem{Pert} B.~Perthame, \it Transport equations in biology, \rm Birkh\"auser Verlag, Basel, 2007

\bibitem{PVW} B.~Perthame, N.~Vauchelet and Z.~Wang, \it
The flux limited Keller-Segel system; properties and
derivation from kinetic equations, \rm  Rev. Mat. Iberoam. {\bf 36} (2020), no. 2, 357--386

\bibitem{Raltetal} D. Ralt et al, \it Chemotaxis and chemokinesis of human spermatozoa to follicular factors,
\rm Biol. Reprod. {\bf 50}, 774--785

\bibitem{RZ} J.A.~Riffel and R.K.~Zimmer, \it Sex and flow: the consequences of fluid shear for sperm-egg interactions,
\rm The Journal of Experimental Biology {\bf 210} (2007), 3644--3660

\bibitem{Taub}
D.D.~Taub et al, \it Monocyte chemotactic protein-1 (MCP-1), -2, and -3 are chemotactic for human T lymphocytes, \rm
J Clin Invest. {\bf 95} (1995), no 3, 1370--1376

\bibitem{JV} J.J.L.~Velazquez, \it Stability of some mechanisms of chemotactic aggregation, \rm SIAM J. Appl. Math. {\bf 62} (2002), 1581--1633

\bibitem{Winkler2021} M. Winkler, \it Suppressing blow-up by gradient-dependent flux limitation in a planar Keller-Segel-Navier-Stokes system,
\rm  Z. Angew. Math. Phys. {\bf 72} (2021), no. 2, Paper No. 72, 24 pp

\bibitem{Winkler2021a} M. Winkler, \it Global weak solutions in a three-dimensional Keller-Segel-Navier-Stokes system with gradient-dependent flux limitation,
\rm Nonlinear Anal. Real World Appl. {\bf 59} (2021), Paper No. 103257, 22 pp.

\bibitem{Yund}  P. Yund, \it How severe is sperm limitation in natural populations of marine free-spawners? \rm Trends Ecol.
Evol. {\bf 15} (2000), 10--14


\bibitem{ZR} R.K.~Zimmer and J.A.~Riffel, \it Sperm chemotaxis, fluid shear, and the evolution
of sexual reproduction, \rm Proc. Nat. Acad. Science {\bf 108} (2011), 13200--13205

\end{thebibliography}
\end{document}